\documentclass[11pt]{amsproc}   

\newtheorem{theorem}{Theorem}[section] 
\newtheorem{proposition}[theorem]{Proposition}
\newtheorem{lemma}[theorem]{Lemma}
\newtheorem{corollary}[theorem]{Corollary} 
\theoremstyle{definition}
\newtheorem{definition}[theorem]{Definition}  
\theoremstyle{remark}
\newtheorem{remark}[theorem]{Remark}

\newtheorem{axiom}{Axiom}[section]

\usepackage[all]{xy}   	 
\usepackage{amsmath,amsthm,amscd,amsfonts,amssymb,enumerate,mathptmx,newtxtext}      
\usepackage{mathtools}

\usepackage[centering]{geometry}   

\usepackage{xstring}
\usepackage{xparse}
\usepackage{xr-hyper}
\usepackage{xcolor}
\definecolor{brightmaroon}{rgb}{0.76, 0.13, 0.28}
\usepackage[linktocpage=true,colorlinks=true,hyperindex,citecolor=blue,linkcolor=brightmaroon]{hyperref} 

\makeatletter
\def\msection{\@startsection{section} 
{1} 
{0pt} 
{-3ex plus -.1ex minus -0.9ex} 
{-.9ex plus -.2ex} 
{\bfseries} 
}

\def\msubsection{\@startsection{subsection} 
{2} 
{0pt} 
{-3ex plus -.1ex minus -0.2ex} 
{-.9ex plus -.2ex} 
{\normalfont\bfseries} 
}
\makeatother

\begin{document}

\author{Amartya Goswami}

\address{Department of Mathematics and Applied Mathematics,\\University of Limpopo, Sovenga 0727, South Africa}

\email{amartya.goswami@ul.ac.za}

\title{Salamander lemma for non-abelian group-like structures}

\begin{abstract}
It is well known that the classical diagram lemmas of homological algebra for abelian groups can be generalized to non-abelian group-like structures, such as groups, rings, algebras, loops, etc. In this paper we establish such a generalization of the ``salamander lemma'' due to G.~M.~Bergman, in a self-dual axiomatic context (developed originally by Z.~Janelidze), which applies to all usual non-abelian group-like structures and also covers axiomatic contexts such as semi-abelian categories in the sense of G.~Janelidze, L.~M\'arki and W.~Tholen and exact categories in the sense of M.~Grandis.  
\end{abstract}

\makeatletter
\@namedef{subjclassname@2020}{%
\textup{2020} Mathematics Subject Classification}
\makeatother



\maketitle 

\msection{Introduction}\label{sec-Introduction}

The salamander lemma is a diagram lemma on a double complex formulated for abelian categories in \cite{B12}, where it has been shown that the other diagram lemmas of homological algebra (specifically, the $3\times 3$ lemma, the four lemma, the snake lemma, Goursat theorem, and the lemma on the long exact sequence
of homology associated with a short exact sequence of complexes) can be recovered from the salamander lemma. In this paper we formulate and prove a non-abelian version of the salamander lemma, in a self-dual axiomatic context presented in \cite{GJ17}, which includes all semi-abelian categories \cite{JMT02} and Grandis exact categories \cite{G84,G92,G12}. These categories are two separate generalizations of abelian categories, which in turn include many important non-abelian categories of group-like structures. Hence the context in which we prove the non-abelian salamander lemma can be applied to groups, rings, loops, Lie algebras, modules and vector spaces in particular, projective spaces, graded abelian groups, and many others.  In this paper, the proof of the salamander lemma reduces to a proposition on exactness of a sequence of subquotients (see Proposition~\ref{ext1} below).  

Two corollaries, (namely Corollary 2.1 and Corollary 2.2) of the salamander lemma in \cite{B12} have been used to recover above mentioned diagram lemmas of homological algebra. We give a reformulation of these two corollaries in the present context, and as an example we  apply them to prove $3\times 3$ lemma.   

\msection{The context}

After almost seventy years since S.~Mac~Lane had the idea (see \cite{M50}) to revisit basic homomorphism theorems of groups in a self-dual axiomatic framework, it is in \cite{GJ17} where a convincing framework has been described in order to achieve this goal.

The case of abelian groups led to the notion of an \emph{abelian category} (refined by Buchsbaum in \cite{B55}), which completely addressed the problem in the abelian case. After the work  of Grothendieck \cite{G57}, this became the central context for homological algebra. Duality allows to get two dual results out of one. For non-abelian groups, some developments (see \cite{W66,W71}) have been made, but without major success in respect to duality.  Instead, a non-dual category-theoretic treatment of groups and group-like structures flourished, which culminated with introduction \emph{semi-abelian category}, introduced in \cite{JMT02}. The context of semi-abelian categories allows a unified treatment (see e.g.~\cite{BB04}) of all standard homomorphism theorems (i.e. isomorphism theorems and diagram lemmas of homological algebra).    

The paper \cite{GJ17} shows that the difficulties in expressing duality phenomenon for group-like structures can be overcome by using functorial duality in the place of categorical duality. This idea originated in the  work \cite{J14} (see also \cite{JW14,  JW16, JW17}) on the comparison between semi-abelian categories with those appearing in the work of Grandis in homological algebra \cite{G84,G92,G12,G13}.   
 
The ``self-dual theory'' of \cite{GJ17} is based on five self-dual axioms which we recall in the next section after introducing the necessary language of this set up. Among the consequences of these axioms, we mention only a few (Lemma \ref{A} - Lemma \ref{rml})  that required in proving the salamander lemma (in this  context) and for details about these and other consequences, we refer our readers to \cite{GJ17}.  
  
In this section we briefly recall the axiomatic context introduced in \cite{GJ17}. This context consists of abstract objects, called ``groups'', which in concrete cases could be groups, rings, modules, or some other group-like structures. The abstract objects form a category whose maps are called ``morphisms'' (in the case of a particular type of group-like structures these are the usual morphisms of those structures). For each group, there is a specified bounded lattice of ``subgroups'', whose partial order is written as ``$\subseteq$'' (again, in the general context these lattices are given abstractly, and in concrete contexts they are the usual substructure lattices). To each morphism $f\colon G\to H$ there is an associated Galois connection between subgroup lattices 
\newcommand\mapsfrom{\mathrel{\reflectbox{\ensuremath{\mapsto}}}} 
\begin{align*}
\mathtt{Sub}\;\!G&\to \mathtt{Sub}\,\!H,\\
S&\mapsto f S,\\
f^{-1} T&\mapsfrom T, 
\end{align*}
which for concrete group-like structures is the Galois connection between substructure lattices given by the  direct and inverse images of substructures along the morphism $f$ (In the general case, we use the same terminology and call $f S$ the direct image of $S$ under $f$ and $f^{-1} T,$ the inverse image of $T$ under $f$). This data is subject to axioms recalled below.  The axioms are invariant under duality, which extends the usual categorical duality and is summarised by the following table (it is in fact an instance of a ``functorial duality'' as explained in \cite{GJ17}):
$$\xymatrix@R=2pt{ \textrm{Expression} & \textrm{Dual Expression} \\ \textrm{$G$ is a group} & \textrm{$G$ is a group} \\ \textrm{$S\in\mathtt{Sub}\;\!G$} & \textrm{$S\in\mathtt{Sub}\;\!G$}\\ \textrm{$S\subseteq T$ in $\mathtt{Sub}\;\!G$} & \textrm{$T\subseteq S$ in $\mathtt{Sub}\;\!G$} \\ f\colon G\to H & f\colon H\to G \\ f g & g f \\ f S & f^{-1} S\\ f^{-1} T & f T. }$$ 
In this context, for a group $G$ by $1$ we denote the bottom element of $\mathtt{Sub}\,G$ (and we call it the \emph{smallest subgroup} of $G$), and by $G$ we denote its top element (calling it the \emph{largest subgroup} of $G$). The \emph{image} of a group morphism $f\colon G\to H$  is defined as $\mathtt{Im}\,\!f=f G$. The dual notion is that of a \emph{kernel} of a group morphism, $\mathtt{Ker}\,\!f=f^{-1} 1$. When $G= \mathtt{Ker}\,\! f,$ we call $f$  a \emph{zero morphism} and  denote it by $0.$ The \emph{identity morphism} $1_G\colon G\to G$ for a group $G,$ is the morphism such that $1_Gf=f$ and $g1_G=g$ for arbitrary morphisms $f\colon F\to G$ and $g:G\to H.$ 
An \emph{isomorphism} is a morphism $f\colon X\to Y$ such that $fg=1_Y$ and $gf=1_X$ for some morphism $g\colon Y\to X.$ A \emph{normal subgroup} of a group $G$ is its subgroup $S$ which is the kernel of some group morphism $f\colon G\to H$ and dually, a \emph{conormal subgroup} $S$ of a group $G$ is a subgroup of $G$ which appears as the image of some group morphism $f\colon F\rightarrow G.$ In standard examples, all subgroups are conormal. In the general theory, however, we do not want to require this since its dual would force all subgroups to be normal. The axioms of our ``self-dual theory'' are as follows: 
\begin{axiom}\label{ax1}	
Assigning to  each group morphism $f\colon G\to H$ the  Galois connection 
$$\xymatrix{ \mathtt{Sub}\;\!G\ar@<2pt>[r] & \mathtt{Sub}\,\!H\ar@<2pt>[l]}$$
given by direct and inverse image maps under $f,$  defines a functor from the category of groups to the category of posets and Galois connections. 
\end{axiom}
\begin{axiom}\label{ax2}
For any group morphism $f\colon G\to H$ and subgroups $A$ of $G$ and $B$ of $H$ we have $ff^{-1} B=B\wedge \mathtt{Im}\,\!f$ and $f^{-1}f A=A\vee \mathtt{Ker}\,\!f.$ 
\end{axiom}
\begin{axiom} \label{ax3}
Each conormal subgroup $S$ of a group $G$ admits an embedding $\iota_S\colon S/1 \to G$ such that $\mathtt{Im}\iota_S\subseteq S$ and for arbitrary group morphism $f\colon U\to G$ such that $\mathtt{Im}f\subseteq S,$ we have $f=\iota_Su$ for a unique homomorphism $u\colon U\to S/1$. Dually, each normal subgroup $S$ of a group $G$ admits a projection $\pi_S\colon G\to G/S$ such that $S\subseteq\mathtt{Ker}\pi_S$ and for an arbitrary group homomorphism $g\colon G\to V$ such that $S\subseteq\mathtt{Ker}g,$ we have $g=v\pi_S$ for a unique homomorphism $v\colon G/S\to V.$ 
\end{axiom}

In classical group theory, Axiom~\ref{ax3} tells us that $\iota_{S}$ is the embedding of the group $S$ into the group $G,$ and $\pi_{S}$ is the quotient map from $G$ to the quotient of $G$ by the normal subgroup generated by $S.$ 

Back in the general context, a subgroup $B$ of a group $G$ is said to be \emph{normal to} a subgroup $A$ of $G$ when (i) $B\subseteq A,$ (ii) $A$ is a conormal subgroup of $G,$ and (iii) $\iota_{A}^{-1} B$ is a normal subgroup of the domain of $\iota_{A}^{-1}$. When a subgroup $B$ is normal to a conormal subgroup $A,$ we denote the codomain of $\pi_{\iota_{A}^{-1}B}$ as $A/B.$ We also write $B\triangleleft A$ when this relation holds.

\begin{axiom}	\label{ax4}
Any group morphism $f\colon G\to H$ factorizes as $f=\iota_{\mathtt{Im}\,\!f}h\pi_{\mathtt{Ker}\,\!f}$ where $h$ is an isomorphism.
\end{axiom}
\begin{axiom}	\label{ax5}
The join of any two normal subgroups of a group is normal and the meet of any two conormal subgroups is conormal.
\end{axiom}
Recall from \cite{GJ17} that among the consequences of the axioms above are the following lemmas.

\begin{lemma}	\label{A}
The direct image map will always preserve joins of subgroups and the inverse image map will always preserve meets of subgroups. 
\end{lemma}
\begin{lemma}\label{A1}
Any embedding is a monomorphism, i.e.~if $mu=mu'$ then $u=u',$ for any embedding $m\colon M\to G$ and any pair of parallel homomorphisms $u$ and $u'$ with codomain $M.$
\end{lemma}

\begin{lemma}	\label{B}
The embedding of an image has trivial kernel and dually, the image of a projection is the largest subgroup of its codomain.
\end{lemma}
\begin{lemma}\label{B1}
 A morphism is both an embedding and a projection if and only if it is an isomorphism.
\end{lemma}
\begin{lemma}	\label{C}
Whenever $A\vee B\subseteq S,$ where $S$ is conormal in some group $G$, we have: $\iota_{S}^{-1}(A\vee B)=\iota_{S}^{-1} A\vee \iota^{-1}_{S} B.$
\end{lemma}
\begin{lemma}\label{B2}
 Normal subgroups are stable under direct images along projections and conormal subgroups are stable under inverse images along embeddings.  
\end{lemma}
\begin{lemma}[Restricted Modular Law]\label{rml}
For any three subgroups $X,$ $Y,$ and $Z$ of a group $G,$ if either $Y$ is normal and $Z$ is conormal, or $Y$ is conormal and $X$ is normal, then we have:
$X\subseteq Z \; \Rightarrow \; X\vee(Y\wedge Z)= (X\vee Y)\wedge Z.$
\end{lemma}
\msection{Exact sequences of subquotients}
\begin{proposition}	\label{pro1}
Let $f\colon G\to H$ be a group morphism and let $Y\vartriangleleft X$ be subgroups of $G.$ Let $\smash{V\vartriangleleft U}$  be subgroups of $H$. If $f Y\subseteq V$ and $fX\subseteq U$ then there is a  morphism $f'\colon X/Y \to U/V$ such that for any subgroup $S$ of $X/Y,$ we have
\begin{equation}	\label{conn1}
f' S=  \pi_{\iota^{-1}_U V} \iota^{-1}_Uf\iota_X\pi^{-1}_{\iota^{-1}_X Y} S.
\end{equation}
\end{proposition}
\begin{remark}
The right hand side of the identity (\ref{conn1}) represents the result of ``chasing'' a subgroup $S$ of $X/Y$ along the zigzag of solid morphisms in the following diagram:
\begin{equation}	\label{ladder}
\vcenter{
\xymatrix@=2pc{G\ar[r]^{f} & H\\
X/1\ar[u]^{\iota_X}\ar@{..>}[r]_{f''}\ar[d]_{{\pi_{\iota^{-1}_X Y}}} & U/1\ar[u]_{\iota_U}\ar[d]^{\pi_{\iota^{-1}_U V}}\\
X/Y\ar@{..>}[r]_{f'} & U/V.
}}
\end{equation}
\end{remark}

\begin{proof}
Since $f X \subseteq U,$ by the universal property of $\iota_U,$ there exists a unique morphism $f''\colon X/1 \to U/1$ such that the top square of diagram (\ref{ladder}) commutes.
Since for any subgroup $S$ of $X/1,$ $\iota^{-1}_U\iota_U f'' S=f'' S\vee \mathtt{Ker}\,\!\iota_U=f'' S$ (where the triviality of $\mathtt{Ker}\,\!\iota_U$  follows from Lemma \ref{B}), we have $\iota^{-1}_Uf\iota_X S=f''S.$  From $Y \subseteq f^{-1} V$ we obtain $$\iota^{-1}_X Y \subseteq \iota^{-1}_{X} f^{-1}V= f''^{-1}\iota^{-1}_{U}V  \subseteq f''^{-1} \mathtt{Ker}\,\!\pi_{\iota^{-1}_U V} = \mathtt{Ker}\,\!\pi_{\iota^{-1}_U V} f'',$$ and by the universal property of $\pi_{\iota^{-1}_X Y},$ there exists a unique morphism $f'\colon X/Y \to U/V$ such that
the bottom square of (\ref{ladder}) commutes. Finally, for any subgroup $S$ of $X/Y,$  we have $$f' S=f'\pi_{\iota^{-1}_X Y}\pi^{-1}_{\iota^{-1}_X  Y} S=\pi_{\iota^{-1}_U V} f''\pi^{-1}_{\iota^{-1}_X Y }S=\pi_{\iota^{-1}_U V} \iota^{-1}_Uf\iota_X\pi^{-1}_{\iota^{-1}_X Y} S,$$ 
where the first equality follows from Lemma \ref{B} and Axiom \ref{ax2}. 
\end{proof}

By taking $f$ in Propostion \ref{pro1} to be an identity morphism, we obtain:

\begin{corollary}	\label{cor1}		
Let $G$ be a group and let $X,$ $Y,$ $U,$ $V$ be subgroups of $G$ such that $Y\vartriangleleft X$ and $V\vartriangleleft U.$ If $Y \subseteq V$ and $X \subseteq U$ then there exists a  morphism $f'\colon X/Y \to U/V$ such that for any subgroup $S$ of $X/Y,$ we have
$f' S=  \pi_{\iota^{-1}_U V} \iota^{-1}_U\iota_X\pi^{-1}_{\iota^{-1}_X Y} S.$
\end{corollary}
\begin{definition}
A sequence $G \xrightarrow {f} H\xrightarrow{g} I$ of group morphisms is called \emph{exact at $H$} if $\mathtt{Im}\,\!f=\mathtt{Ker}\,\!g.$
\end{definition}
\begin{proposition}\label{ext1}
Let $G \xrightarrow {f} H\xrightarrow{g} I$  be  group morphisms. Let  $V\vartriangleleft U$ be subgroups of $G,$ let $X\vartriangleleft W$ be subgroups of $H,$ and let $Z\vartriangleleft Y$ be subgroups of $I.$ Suppose $f V\subseteq X,$ $fU \subseteq W,$ $ g X\subseteq Z$ and $ gW \subseteq  Y.$ Then, there is a sequence 
\begin{equation}	\label{seq}
U/V\to W/X\to Y/Z
\end{equation} of  morphisms 
which is exact at $W/X$ if and only if 
\begin{equation}\label{ext}
f U  \vee X = g^{-1} Z\wedge W.
\end{equation}
\end{proposition}
\begin{proof}
Let us consider the following diagram:
\begin{equation*}\label{ladder2}
\vcenter{
\xymatrix@=2pc{G\ar[r]^{f} & H\ar[r]^{g} & I\\
U/1\ar[u]^{\iota_U}\ar@{..>}[r]\ar[d]_{{\pi_{\iota^{-1}_U V}}} & W/1\ar@{..>}[r]\ar[u]_{\iota_W}\ar[d]^{\pi_{\iota^{-1}_W X}}& Y/1\ar[u]_{\iota_{Y}}\ar[d]^{\pi_{\iota^{-1}_Y Z}}\\
U/V\ar@{..>}[r] & W/X\ar@{..>}[r]& Y/Z.
}}  
\end{equation*}
The existence of the group morphisms of the sequence (\ref{seq}) follows from Proposition \ref{pro1}. 

Let us assume that the identity (\ref{ext}) holds and we will prove the exactness at $W/X$ and for that it is sufficient to show the image of $U/V\to W/X$ is equal to the kernel of $W/X\to Y/Z.$ We observe that the image of $U/V\to W/X$ is the image of the largest subgroup of $U/V,$ which by Proposition  \ref{pro1} is same chasing the largest subgroup of $U/V$ along the zigzag of solid arrows in the diagram above up to $W/X.$ By doing the chasing, we obtain 
\begin{align*}
 \pi_{\iota^{-1}_W X} \iota^{-1}_Wf \iota_U \pi_{\iota^{-1}_UV}U/V &= \pi_{\iota^{-1}_W X} \iota^{-1}_Wf \iota_U U\\
 &=\pi_{\iota^{-1}_W X} \iota^{-1}_Wf U\\
 &= \pi_{\iota^{-1}_W X} \iota^{-1}_Wf U \vee \pi_{\iota^{-1}_W X} \iota^{-1}_W X \qquad\qquad   [\mathtt{Ker}\,\!\pi_{\iota^{-1}_W X} \supseteq \iota^{-1}_W X]\\
 &=\pi_{\iota^{-1}_W X} (\iota^{-1}_W f U \vee \iota^{-1}_W X) \qquad\qquad\qquad\qquad\;\; [\mathrm{Lemma}\;\ref{A}]\\
 &=\pi_{\iota^{-1}_W X} \iota^{-1}_W (f U \vee X) \qquad\qquad\qquad\qquad\;\quad\;\;\, [\mathrm{Lemma}\;\ref{C}] 
\end{align*}
Similarly, the kernel of $W/X\to Y/Z$ is the inverse image of the smallest subgroup of $Y/Z,$ which by Proposition \ref{pro1} is same as chasing the smallest subgroup of $Y/Z$ along the zigzag of the solid arrows in the diagram above up to $W/X.$ By doing so, we get 
\begin{align*}
\pi_{\iota^{-1}_WX}\iota^{-1}_Wg^{-1}\iota_Y\pi^{-1}_{\iota^{-1}_YZ} 1 &= \pi_{\iota^{-1}_WX}\iota^{-1}_Wg^{-1}\iota_YZ\\
&= \pi_{\iota^{-1}_WX}\iota^{-1}_Wg^{-1}Z\\
&=\pi_{\iota^{-1}_WX}\iota^{-1}_W(g^{-1}Z\wedge W),
\end{align*}
where the reason of the last equality is as follows:
$$ \iota^{-1}_Wg^{-1}Z=\iota^{-1}_Wg^{-1}Z \vee 1= \iota^{-1}_W\iota_W(\iota^{-1}_Wg^{-1}Z)=  \iota^{-1}_W(\iota_W\iota^{-1}_Wg^{-1}Z)=\iota^{-1}_W(g^{-1}Z\wedge W).$$
The identity (\ref{ext}), and the two outcomes of the above chasing give the desired exactness at $W/X.$ 

Conversely,  let us assume that  (\ref{seq}) be exact at $W/X$, i.e.  in particular, we have  $\pi_{\iota^{-1}_W X} \iota^{-1}_Wf U = \pi_{\iota^{-1}_WX} \iota^{-1}_W (g^{-1} Z\wedge W).$  Applying $\iota_W\pi^{-1}_{\iota^{-1}_W X}$ on the left and the right hand sides of the last identity, we obtain, respectively 
\begin{align*}
\iota_W\pi^{-1}_{\iota^{-1}_W X}\pi_{\iota^{-1}_W X} (\iota^{-1}_Wf U) &= \iota_W(\iota_W^{-1}f U \vee \iota^{-1}_W X)\qquad\qquad\;\quad [\textrm{ Axiom}\;\ref{ax2}]\\
&=\iota_W\iota^{-1}_W(f U\vee X)\qquad\qquad\quad\;\quad\;[\mathrm{ Lemma}\;\ref{C}]\\
&= (f U\vee X)\wedge W\qquad\qquad\qquad\quad\;\; [\textrm{Axiom}\;\ref{ax2}]\\
&= f U\vee X,\qquad\qquad\;\quad\; [\mathrm{hypothesis,}\; fU\subseteq W]\\
\noalign{\hbox{and}}
\iota_W\pi^{-1}_{\iota^{-1}_W X}\pi_{\iota^{-1}_W X} (\iota^{-1}_W(g^{-1}Z\wedge W)) &= \iota_W(\iota_W^{-1}(g^{-1} Z\wedge W) \vee \iota^{-1}_W X)\quad\;\; [\textrm{Axiom}\;\ref{ax2}]\\
&=\iota_W\iota^{-1}_W(g^{-1} Z\wedge W)\quad\qquad [X\subseteq W, gX\subseteq Z,\\&\qquad\qquad\qquad \qquad \qquad\qquad  \;\;\;\; \mathrm{and \; Lemma}\;\ref{C}]\\ \\  
&= (g^{-1} Z\wedge W)\wedge W\qquad\quad\qquad\quad\hfill [\textrm{ Axiom}\;\ref{ax2}]\\
&= g^{-1} Z\wedge W.
\end{align*}
\end{proof}

\msection{Double complexes and salamander lemma} 
\begin{definition}
A \emph{double complex} is a triple $(X, \delta_h, \delta_v),$ where for all integers $m$ and $n,$ $X=(X^{n,m})$ is a family of groups,  $\delta_h=(\delta_h^{n,m}\colon X^{n,m}\to X^{n,m+1}),$ and $\delta_v=(\delta_v^{n,m}\colon X^{n,m}\to X^{n+1,m})$ are families of group morphisms such that
$\delta_h^{n,m}\delta_h^{n,m-1}=0,\;$ $\delta_v^{n,m}\delta_v^{n-1,m}=0,\;$and $\;\delta_v^{n,m+1}\delta_h^{n,m}=\delta_h^{n+1,m}\delta_v^{n,m}.$
\end{definition}
Let us consider a double complex as shown in the diagram
\begin{equation}\label{sl}
\vcenter{
\xymatrix@C=20pt@R=16pt{ &&\ar@{.}[d]  &\ar@{.}[d]\\&& \bullet\ar[d]^{m} & \ar[d]\\
\ar@{.}[r] &\bullet\ar[r]^{a}\ar[dr]^{p} & C \ar[dr]^{r} \ar[r]\ar[d]^{c} & \bullet\ar[d]^{v}\ar[r] & \ar@{.}[r] &\\
\ar@{.}[r] &	\bullet \ar[r]_{d} & A \ar[r]^{e}\ar[d]_{f} \ar[dr]^{\!q} & B\ar[r]^{s} \ar[d]^{g}& \bullet \ar@{.}[r] &\\
\ar@{.}[r] & \ar[r]	&\bullet\ar[r]\ar[d] & D\ar[r]^{t}\ar[d]^{u} & \bullet\ar@{.}[r] &\\
& &  \ar@{.}[d] & \bullet\ar@{.}[d] &\\
&&&&
}}
\end{equation}
where $p=ca,$ $r=ec,$ and $q=ge.$ 
\begin{definition} 	\label{ho1}
Following Bergman \cite{B12}, in a double complex (\ref{sl}) we define the following \emph{homology objects} associated with the group $A$:  
\begin{itemize}
\item [$\bullet$] $\displaystyle A_{\mathtt{h}} = \mathtt{Ker}\,\! e/\mathtt{Im}\,\! d,$ whenever $ \mathtt{Im}\,\! d \vartriangleleft \mathtt{Ker}\,\! e;$
\item [$\bullet$] $\displaystyle A_{\Box} = \mathtt{Ker}\,\! q/(\mathtt{Im}\,\! c \vee \mathtt{Im}\,\!d),$ whenever $(\mathtt{Im}\,\!c \vee \mathtt{Im}\,\!d) \vartriangleleft \mathtt{Ker}\,\!q;$
\item [$\bullet$] $\displaystyle ^\Box\!\!A = (\mathtt{Ker}\,\! e \wedge \mathtt{Ker}\,\!f)/\mathtt{Im}\,\! p,$ whenever $\mathtt{Im}\,\!p  \vartriangleleft (\mathtt{Ker}\,\!e \wedge \mathtt{Ker}\,\!f).$ 
\end{itemize}
	
When we say that one of the above three homology objects \emph{is defined}, we mean that the corresponding normality condition holds.
\end{definition} 
\begin{theorem}[Salamander Lemma]	\label{slt} In a double complex \textup{(}\ref{sl}\textup{)},
if the homology objects $C_{\Box},$ $A_{\mathtt{h}},$ $A_{\Box},$ $^\Box\! B,$ $B_{\mathtt{h}},$ and $^\Box\!D$ are defined, and $\mathtt{Im}\,\!c$ is a normal subgroup of $A,$ then there is an exact sequence 
\begin{equation}	\label{sallem}
C_{\Box}\to   A_{\mathtt{h}} \to  A_{\Box} \to ^\Box\!\!\!B\to B_{\mathtt{h}}\to ^\Box\!\!\!D.
\end{equation}
\end{theorem}
\begin{proof} 
For proving the existence of all the  morphisms of the sequence (\ref{sallem}), we check the hypothesises of Proposition \ref{pro1} or Corollary \ref{cor1} whichever is applicable. 
\begin{itemize}
\item [$\bullet$] $C_{\Box}\to A_{\mathtt{h}}$:	
$
c\mathtt{Ker}\,\!r=c\mathtt{Ker}\,\!ec =cc^{-1}\mathtt{Ker}\,\!e=\mathtt{Ker}\,\!e\wedge \mathtt{Im}\,\!c\subseteq \mathtt{Ker}\,\!e,$ and using Lemma \ref{A}, we get 
$c(\mathtt{Im}\,\!a\vee \mathtt{Im}\,\!m)=c\mathtt{Im}\,\!a \vee c\mathtt{Im}\,\!m=\mathtt{Im}\,\!ca \vee 1 = \mathtt{Im}\,\!p \subseteq \mathtt{Im}\,\!d.$
\item  [$\bullet$] $A_{\mathtt{h}}\to A_{\Box}$: $\mathtt{Ker}\,\!e\subseteq \mathtt{Ker}\,\!q$ and $\mathtt{Im}\,\!d \subseteq \mathtt{Im}\,\!c\vee \mathtt{Im}\,\!d.$
\item  [$\bullet$] $ A_{\Box}\to ^\Box\!\!\!B$:
$
e\mathtt{Ker}\,\!q=e\mathtt{Ker}\,\!ge =ee^{-1}\mathtt{Ker}\,\!g=\mathtt{Im}\,\!e\wedge\mathtt{Ker}\,\!g \subseteq \mathtt{Ker}\,\!s \wedge \mathtt{Ker}\,\!g,$ and again using Lemma \ref{A}, we get  $
e(\mathtt{Im}\,\!c\vee \mathtt{Im}\,\!d)=e\mathtt{Im}\,\!c \vee e\mathtt{Im}\,\!d=\mathtt{Im}\,\!r \vee 1 = \mathtt{Im}\,\!r.$
\item [$\bullet$] $^\Box\!B\to B_{\mathtt{h}}$: $\mathtt{Ker}\,\!s \wedge \mathtt{Ker}\,\!g \subseteq \mathtt{Ker}\,\!s$ and $\mathtt{Im}\,\!r\subseteq \mathtt{Im}\,\!e.$
\item [$\bullet$] $ B_{\mathtt{h}} \to ^\Box\!\!\!D$:
$
g\mathtt{Ker}\,\!s \subseteq \mathtt{Ker}\,\!u,$ $g\mathtt{Ker}\,\!s \subseteq g\mathtt{Ker}\,\!tg=gg^{-1}\mathtt{Ker}\,\!t =\mathtt{Im}\,\!g \wedge \mathtt{Ker}\,\!t \subseteq \mathtt{Ker}\,\!t,$ which implies $g\mathtt{Ker}\,\!s\subseteq \mathtt{Ker}\,\!u\wedge \mathtt{Ker}\,\!t.$ Also $
g\mathtt{Im}\,\!e=\mathtt{Im}\,\!q.$
\end{itemize}
For the exactness,  we apply Proposition \ref{ext1}, by checking the condition (\ref{ext}).
\begin{itemize}
\item Exactness of $C_{\Box} \to  A_{\mathtt{h}}\to A_{\Box}$: 
For the left hand side of  (\ref{ext}), we have $c\mathtt{Ker}\,\!r  \vee \mathtt{Im}\,\!d = (\mathtt{Ker}\,\!e \wedge \mathtt{Im}\,\!c)\vee \mathtt{Im}\,\!d,$ whereas the right hand side of (\ref{ext})   is $(\mathtt{Im}\,\!c \vee \mathtt{Im}\,\!d)\wedge \mathtt{Ker}\,\!e= (\mathtt{Ker}\,\!e \wedge \mathtt{Im}\,\!c)\vee \mathtt{Im}\,\!d$ (by Lemma \ref{rml}). 	
\item  Exactness of $A_{\mathtt{h}} \to A_{\Box}\to ^\Box\!\!\!B$: 
We notice that the left hand side of (\ref{ext}) is 
$
\mathtt{Ker}\,\!e \vee (\mathtt{Im}\,\!c \vee \mathtt{Im}\,\!d) 
= \mathtt{Ker}\,\!e \vee \mathtt{Im}\,\!c,
$
whereas the right hand side of (\ref{ext}) is
$
e^{-1}\mathtt{Im}\,\!r \wedge \mathtt{Ker}\,\!q=e^{-1}e\mathtt{Im}\,\!c\wedge \mathtt{Ker}\,\!q= (\mathtt{Ker}\,\!e \vee \mathtt{Im}\,\!c)\wedge \mathtt{Ker}\,\!q =
\mathtt{Ker}\,\!e \vee \mathtt{Im}\,\!c.$
\item  Exactness of $A_{\Box}\to ^\Box\!\!\!B\to B_{\mathtt{h}}$: The left hand side of (\ref{ext}) is $e\mathtt{Ker}\,\!q\vee\mathtt{Im}\,\!r=ee^{-1}\mathtt{Ker}\,\!g\vee \mathtt{Im}\,\!r=(\mathtt{Im}\,\!e\wedge\mathtt{Ker}\,\!g)\vee\mathtt{Im}\,\!r=\mathtt{Im}\,\!e\wedge \mathtt{Ker}\,\!g,$ whereas the right hand side of (\ref{ext}) is $\mathtt{Im}\,\!e\wedge (\mathtt{Ker}\,\!s\wedge \mathtt{Ker}\,\!g)= \mathtt{Im}\,\!e\wedge \mathtt{Ker}\,\!g.$
\item  Exactness of $^\Box\!B\to B_{\mathtt{h}}\to ^\Box\!\!\!D$: The left hand side of (\ref{ext}) is $(\mathtt{Ker}\,\!s\wedge \mathtt{Ker}\,\!g)\vee \mathtt{Im}\,\!e,$ while the right hand side of (\ref{ext}) is $g^{-1}(\mathtt{Im}\,\!q)\wedge \mathtt{Ker}\,\!s = g^{-1}g(\mathtt{Im}\,\!e)\wedge \mathtt{Ker}\,\!s=(\mathtt{Im}\,\!e\vee \mathtt{Ker}\,\!g)\wedge \mathtt{Ker}\,\!s=(\mathtt{Ker}\,\!s\wedge \mathtt{Ker}\,\!g)\vee \mathtt{Im}\,\!e$ (by Lemma \ref{rml}).
\end{itemize}	
\end{proof}
\begin{remark} 
For the proofs of the existence of the morphisms $C_{\Box} \to A_{\mathtt{h}}$ and $B_{\mathtt{h}} \to ^\Box\!\!\!D$ in (\ref{sallem}), we have constructed direct morphisms which are respectively the same as the composites $ C_{\Box}\to ^\Box\!\!\!\!A\to A_{\mathtt{h}}$ and $ B_{\mathtt{h}}\to B_{\Box}\to ^\Box\!\!\!\!\,D$  as have been done in \cite{B12}.
\end{remark}
\begin{remark}
 The Theorem \ref{sl} is the horizontal version of the  salamander lemma. The formulation and proof of the vertical version are similar. 
\end{remark}

The following two corollaries are the reformulation of the Corollary 2.1 and the Corollary 2.2 of \cite{B12}, which are used to prove diagram lemmas of homological algebra. We give a proof of Corollary \ref{cor1} using Proposition \ref{ext1}, and the proof of Corollary \ref{cor2} is similar.
\begin{corollary}\label{cor1} 
Let $A\to B$ be a horizontal \textup{(}vertical\textup{)} morphism  of a double complex. Let  $A_{\mathtt{h}}$ and $B_\mathtt{h}$ \textup{(}$A_{\mathtt{v}}$ and $B_\mathtt{v}$\textup{)} are defined and  $A_{\mathtt{h}}=1,$ $B_\mathtt{h}=1$ \textup{(}$A_{\mathtt{v}}=1$ and $B_\mathtt{v}=1$\textup{)}.  Whenever the homology objects $A_{\Box}$ and $^{\Box}\!B$ are defined, we have the isomorphism:  $A_{\Box} \cong\,\! ^{\Box}\!B.$
\end{corollary}
\begin{proof}
 Let $A\to B$ be the morphism $e$ of the double complex (\ref{slt}). Let $A_{\mathtt{h}}=1$ and $B_\mathtt{h}=1.$ The existence of the morphism $\phi\colon A_{\Box}\to\,\! ^{\Box}\!B$ has been proved in Theorem \ref{sl}. Now to show $\phi$ is an isomorphism, by Lemma \ref{B1}, it is sufficient to show that $\phi$ is both an embedding and a projection. From the double complex (\ref{slt}), we observe that $e\mathtt{Ker}q=e\mathtt{Ker}ge=ee^{-1}\mathtt{Ker}g=\mathtt{Im}e\wedge \mathtt{Ker}g=\mathtt{Ker}b\wedge \mathtt{Ker}g,$ where the last equality follows from the fact that $\mathtt{Im}e=\mathtt{Ker}b.$  
 This proves that $\phi$ is a projection. Again,  
 $e^{-1}\mathtt{Im}r=e^{-1}e\mathtt{Im}c=\mathtt{Ker}e\vee \mathtt{Im}c=\mathtt{Im}\vee \mathtt{Im}d\;(\mathrm{as}\; \mathtt{Im}d=\mathtt{Ker}e)$
 proves that $\phi$ is an embedding.
\end{proof}
\begin{corollary}\label{cor2}
 In each of  the following four portions of  double complexes, if  the dotted row or column
\textup{(}the row or column through $B$ perpendicular to
the arrow connecting it with $A)$
is exact at $B,$ and $A_{\mathtt{h}},$ $A_{\mathtt{v}},$ $^\Box \!\!A,$ $A_{\Box}$ are defined 
 
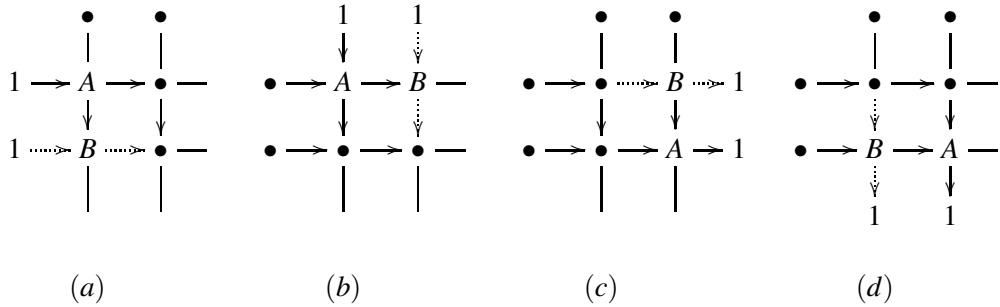
\[       
 \xymatrix@C=.9pc@R=.9pc{
 &\bullet\ar@{-}[d] & \bullet\ar@{-}[d] & & &1\ar[d] & 1\ar@{.>}[d] & & & \bullet\ar@{-}[d] & \bullet\ar@{-}[d] & & &\bullet\ar@{-}[d] & \bullet\ar@{-}[d] &\\
 1\ar[r] &A\ar[r]\ar[d]&\bullet\ar[d]\ar@{-}[r]& & \bullet\ar[r]&A\ar[d]\ar[r] &B\ar@{.>}[d]\ar@{-}[r]& &\bullet\ar[r] &\bullet\ar[d] \ar@{.>}[r]&B\ar[d]\ar@{.>}[r]&1 &\bullet\ar[r] &\bullet\ar@{.>}[d]\ar[r]&\bullet\ar[d]\ar@{-}[r]&\\
 1\ar@{.>}[r]&B\ar@{-}[d]\ar@{.>}[r]&\bullet\ar@{-}[d]\ar@{-}[r]& &\bullet\ar[r]&\bullet\ar@{-}[d]\ar[r]&  \bullet\ar@{-}[d]\ar@{-}[r]& &\bullet\ar[r]& \bullet\ar@{-}[d]\ar[r]&A\ar@{-}[d]\ar[r]&1 & \bullet\ar[r]&B\ar@{.>}[d]\ar[r]&A\ar[d]\ar@{-}[r]& \\   
 & & & & & & & & & & & & & 1&1\\
 &(a) & & & & (b) & & & & (c) & & & & (d) 
 }  
 \]
 then  we have the following pairs of isomorphisms associated  with the above four diagrams respectively:
$(a)\,^\Box\!\!A\cong A_{\mathtt{h}}, \, A_{\mathtt{v}}\cong A_{\Box}; (b)\,^\Box\!\!A \cong A_{\mathtt{v}},\,A_{\mathtt{h}}\cong A_{\Box};$\\ $ (c)\;A_{\mathtt{h}}\cong A_{\Box},\, ^\Box\!\!A\cong A_{\mathtt{v}};\;(d)\;\; A_{\mathtt{v}}\cong A_{\Box},\, ^\Box\!\!A \cong A_{\mathtt{h}}.$ 
\end{corollary} 
\begin{theorem}[$3\times 3$ Lemma]
 In the commutative diagram below, if all columns, and all rows but the first, are exact, then the first row is also exact. 
 \[  
 \xymatrix@C=2pc@R=2pc{
 &1\ar[d]^{y_1}&1\ar[d]^{y_5}&1\ar[d]^{y_9}&\\
 1\ar[r]^{x_1}&A'\ar[r]^{x_2}\ar[d]^{y_2}\ar[dr]^{z_1}&B'\ar[r]^{x_3}\ar[d]^{y_6}\ar[dr]^{z_2}&C'\ar[r]^{x_4}\ar[d]^{\!y_{10}}\ar[dr]^{z_3}&1\\ 
 1\ar[r]^{x_5}&A\ar[r]^{x_6}\ar[d]^{y_3}\ar[dr]^{z_4}&B\ar[r]^{x_7}\ar[d]^{y_7}\ar[dr]^{z_5}&C\ar[r]^{x_8}\ar[d]^{y_{11}}&1\\ 
 1\ar[r]^{x_9}&A''\ar[r]^{x_{10}}\ar[d]^{y_4}\ar[dr]^{z_6}&B''\ar[r]^{x_{11}}\ar[d]^{y_8}&C''\ar[r]^{x_{12}}  \ar[d]^{y_{12}}&1\\  
 &1 &1 &1 
 } 
 \]
\end{theorem}
\begin{proof} To show that the first row is a complex, we notice $ y_{10}x_3x_2=x_7x_6y_2=0,$ and since $y_{10}$ is an embedding, by Lemma \ref{A1} we have $x_3x_2=0.$ 

By the approach of \cite{B12}, to show the trivialities of $A'_{\mathtt{h}},$ $B'_{\mathtt{h}},$ and $C'_{\mathtt{h}},$  we need to consider the following homology objects: 
$$ A'_{\mathtt{h}},\, A'_{\Box},\, A'_{\mathtt{v}},\,  B'_{\mathtt{h}},\, B'_{\Box},\, ^\Box\!B,\, A_{\Box},\,  A_{\mathtt{v}},\, C'_{\mathtt{h}},\, C'_{\Box},\, ^\Box\!C,\, B_{\Box},\, ^\Box\!B'',\, A''_{\Box},\,  A''_{\mathtt{v}}.  $$ 
For them to be defined in self-dual context, we need to verify their respective normality conditions. We show the method of verification for $B'_{\Box},$ and the others can be checked similarly.   

In order to show  $(\mathtt{Im}x_2\vee \mathtt{Im}y_5) \vartriangleleft  \mathtt{Ker}z_2,$ first we observe that $\mathtt{Im}x_2\vee \mathtt{Im}y_5=\mathtt{Im}x_2\vee 1=\mathtt{Im}x_2 \subseteq \mathtt{Ker}x_3\subseteq \mathtt{Ker}z_2.$ To show that  $\mathtt{Ker}z_2$ is a conormal subgroup of $B'$, it is sufficient to show that its dual $\mathtt{Im}z_4$ is a normal subgroup of $B''.$ Now, $\mathtt{Im}z_4=y_7\mathtt{Im}x_6=y_7\mathtt{Ker}x_7$ which is a normal subgroup of $B''$ by Lemma \ref{B2}.  Finally to show that $\iota^{-1}_{\mathtt{Ker}z_2}\mathtt{Im}x_2$ is a normal subgroup of $\mathtt{Ker}z_2/1,$ it is sufficient to show that $\mathtt{Ker}x_{11},$ a dual of $\mathtt{Im}x_2,$    is a conormal subgroup of $B''$ which indeed is true because of the fact that  $\mathtt{Im}x_{10}=\mathtt{Ker}x_{11}.$ 

Applying the Corollary \ref{cor1} and the Corollary \ref{cor2}, the proofs of  trivialities of $A'_{\mathtt{h}},$ $B'_{\mathtt{h}},$ and $C'_{\mathtt{h}}$ are same as in \cite{B12}, and here we recall them.
\begin{align*}
A'_{\mathtt{h}}&\cong A'_{\Box} \cong A'_{\mathtt{v}}=1.\\
B'_{\mathtt{h}}&\cong B'_{\Box}\cong\, ^\Box\!B \cong A_{\Box} \cong A_{\mathtt{v}}=1.\\
C'_{\mathtt{h}}&\cong C'_{\Box}\cong\, ^\Box\!C \cong B_{\Box}\cong\, ^\Box\!B'' \cong A''_{\Box} \cong  A''_{\mathtt{v}}=1. 
\end{align*} 

\end{proof}

\section*{Acknowledgement}
I am grateful to  Zurab Janelidze for many fruitful discussions. 
\appendix

\end{document}